\documentclass[a4paper]{amsart}
\usepackage{graphicx}
\usepackage{amssymb}
\usepackage{mathtools}
\usepackage{stmaryrd}
\usepackage{dsfont}
\usepackage[mathcal]{euscript}
\usepackage{tikz}
\usepackage{tikz-cd}
\usepackage{hyperref}
\hypersetup{%
  bookmarksnumbered=true,%
  colorlinks=true,%
  linkcolor=blue,%
  citecolor=blue,%
  filecolor=blue,%
  menucolor=blue,%
  urlcolor=blue,%
  bookmarksopen=true,%
  bookmarksdepth=2,%
  pageanchor=true}

\title{Locally noetherian quiver representations}

\author{Henning Krause}
\address{Fakult\"at f\"ur Mathematik\\
Universit\"at Bielefeld\\ D-33501 Bielefeld\\ Germany}
\email{hkrause@math.uni-bielefeld.de}

\theoremstyle{plain}
\newtheorem{thm}{Theorem}[section]
\newtheorem{prop}[thm]{Proposition}
\newtheorem{lem}[thm]{Lemma} 

\newtheorem{cor}[thm]{Corollary}

\theoremstyle{definition}
\newtheorem{defn}[thm]{Definition}
\newtheorem{exm}[thm]{Example}

\theoremstyle{remark}
\newtheorem{rem}[thm]{Remark}

\numberwithin{equation}{thm}

\newcommand{\colim}{\operatorname*{colim}}

\newcommand{\Hom}{\operatorname{Hom}}

\newcommand{\Lex}{\operatorname{Lex}}

\renewcommand{\min}{\operatorname{min}}

\newcommand{\Mod}{\operatorname{Mod}}

\newcommand{\Noeth}{\operatorname{Noeth}}

\newcommand{\one}{\mathds 1}

\newcommand{\Rep}{\operatorname{Rep}}

\newcommand{\Ab}{\mathrm{Ab}}

\newcommand{\op}{\mathrm{op}}

\newcommand{\Set}{\mathrm{Set}}

\newcommand{\iso}{\xrightarrow{\raisebox{-.4ex}[0ex][0ex]{$\scriptstyle{\sim}$}}}

\newcommand{\longiso}{\xrightarrow{\ \raisebox{-.4ex}[0ex][0ex]{$\scriptstyle{\sim}$}\ }}

\newcommand{\lto}{\longrightarrow}

\newcommand*{\intref}[2]{\def\tmp{#1}\ifx\tmp\empty\hyperref[#2]{\ref*{#2}}\else\hyperref[#2]{#1~\ref*{#2}}\fi}

\def\A{\mathcal A} 
 
\def\C{\mathcal C}

\def\Q{\mathcal Q}

\def\X{\mathcal X}

\def\bbN{\mathbb N}

\def\a{\alpha}

\def\p{\phi}

\begin{document}



\begin{abstract}
  It is shown that a quiver is left noetherian if and only if the category
  of quiver representations in any locally noetherian abelian category
  is again locally noetherian. Here, locally noetherian means that
  any object is the directed union of its noetherian subobjects. For a
  quiver to be left noetherian means that the left ideals of paths starting at any
  fixed vertex satisfy the ascending chain condition. The proof
  generalises to representations of any small category that admits a
  Gröbner enrichment.
\end{abstract}

\date{August 10, 2024}

\maketitle

\section{Introduction}

Hilbert's basis theorem asserts that a polynomial ring $A[X]$ in one
indeterminate is left noetherian if the coefficient ring $A$ is left
noetherian \cite{Hi1890}.  The category of $A[X]$-modules identifies
with the category of $A$-linear representations of the \emph{Jordan
  quiver} consisting of one vertex and a single loop.  Thus an
equivalent formulation of Hilbert's theorem is that the category of
$A$-linear representations of the Jordan quiver is locally noetherian
if the category of $A$-modules is locally noetherian.

We generalise Hilbert's  theorem as follows and write
$\Rep(Q,\A)$ for the category of representation of a quiver $Q$ in an
abelian category $\A$.

\begin{thm}\label{th:repnoeth}
 For a quiver $Q$ and a left noetherian ring $A$ the following conditions are equivalent.
\begin{enumerate}
\item The quiver $Q$ is left noetherian.
\item The abelian category $\Rep(Q,\A)$ is locally noetherian for every 
  locally noetherian abelian category $\A$.
\item The abelian category $\Rep(Q,\Mod A)$ is locally noetherian. 
\end{enumerate}
\end{thm}

For a quiver to be \emph{left noetherian} simply means that the left
ideals of paths starting at any fixed vertex satisfy the ascending
chain condition, and the quiver is \emph{right noetherian} if the
analogous condition for paths terminating at any fixed vertex is
satisfied; see Definition~\ref{de:noeth}.

When the quiver $Q$ has only finitely many vertices the category
$\Rep(Q,\Mod A)$ identifies with the module category of the \emph{path
  algebra} $A[Q]$. Thus we obtain for $A[Q]$ a simple characterisation to be
left noetherian; see Corollary~\ref{co:algebra}.

There is an analogue of the above theorem for \emph{locally finite
  categories} of representations. This requires both chain conditions
for ideals in a quiver; see Theorem~\ref{th:locfinite}. Moreover, we
provide a generalisation of the above theorem which covers
representations of any \emph{small category}. This involves the notion
of a \emph{Gröbner enrichment} and the proof remains basically
unchanged; see Theorem~\ref{th:repnoeth-general}.

Theorem~\ref{th:repnoeth} has some predecessors, but in this form it
seems to be new, because we give a simple definition/description of a
noetherian quiver and we consider representations in any locally
noetherian abelian category. In fact, our definition of a locally
noetherian abelian category is more general than the usual one
\cite[II.4]{Ga1962}; it simply requires that any object is the
directed union of its noetherian subobjects. For example, the category
of finite dimensional vector spaces over a field is locally
noetherian. We consider representations in abelian categories that are
not necessarily module categories; this requires some new techniques
since elements cannot be used.

Let us provide the relevant references that contain predecessors of
the above theorem. The first is \cite{HR1983} where it is assumed that
the quiver has a finite set of vertices, while \cite{Ri1986} considers
representations of arbitrary quivers and even more general
categories. The description of noetherian quivers in
\cite{HR1983,Ri1986} is somewhat complicated. So the results from
\cite{Ri1986} are reproved in \cite{En2002}, using a more transparent
(though incorrect, as pointed out in \cite{AHV2015}) description of
noetherian quivers. All mentioned papers restrict to representations
in module categories. For a survey on more recent generalisations
of Hilbert's basis theorem, including their applications, see \cite{Dj2016}.

\section{Noetherian quivers}

Let $Q=(Q_0,Q_1,s,t)$ be a \emph{quiver}, given by a set of vertices $Q_0$
and a set of arrows $Q_1$ such that each arrow $\a$ starts at $s(\a)$
and terminates at $t(\a)$.  A \emph{path} $\p\colon x\to y$ of
\emph{length} $n$ in $Q$ is a sequence $\a_n\ldots\a_2\a_1$ of arrows
$\a_i$ such that $s(\a_{i+1})=t(\a_i)$ for all $i$, where
$s(\p):=x=s(\a_1)$ and $t(\p):=y=t(\a_n)$. For each $x\in Q_0$ there is
the \emph{trivial path} $x\to x$ of length zero. For any pair of
vertices $x,y$ in $Q$ let $Q(x,y)$ denote the set of paths $x\to
y$. We write $\Q$ for the \emph{path category} with set of objects
$Q_0$ and set of morphism $x\to y$ given by $Q(x,y)$ for all
$x,y\in Q_0$. The concatenation of paths yields the composition in
$\Q$.

For a vertex $x$ in $Q$ let $Q(x)$ denote the set of paths starting at
$x$. This set is partially ordered via
\[\p\le \psi\quad :\iff \quad  \p=\chi\psi \;\text{ for some path }\; \chi\colon
  t(\psi)\to t(\p).\]

Let $P=(P,\le)$ be a poset. An \emph{ideal} of $P$ is a non-empty
subset $I\subseteq P$ such that $x\le y$ and $y\in I$ implies
$x\in I$. The poset $P$ is \emph{noetherian} if the ascending chain
condition holds for the ideals in $P$. We refer to
\cite[VIII.1]{Bi1948} for basic facts about noetherian
posets.\footnote{Sometimes a poset is defined to be \emph{noetherian}
  if the ascending chain condition holds for chains of
  \emph{elements}. Clearly, one condition implies the other (by
  passing from an element to the ideal generated by that element),
  but the conditions are not equivalent.}  For instance, it is easily checked
that a poset is noetherian if and only if every ideal is generated by
finitely many elements. For later use we record another well known
characterisation.

\begin{lem}\label{le:poset}
  A poset is noetherian if and only if for every infinite sequence
  $(x_i)_{i\in\bbN}$ of elements there is a map $\nu\colon\bbN\to\bbN$
  such that $i<j$ implies $\nu(i)<\nu(j)$ and $x_{\nu(j)}\le
  x_{\nu(i)}$.
\end{lem}
\begin{proof}
  We prove the direction which is used later on and leave the other
  one to the interested reader. So assume  the poset is
  noetherian.  First oberserve that for every infinite sequence
  $(x_i)_{i\in\bbN}$ of elements there exists $i\in\bbN$ such that
  $x_j\le x_i$ for infinitely many $j\in\bbN$. This follows by looking
  at the ascending chain of ideals generated by the elements
  $x_0,\ldots,x_n$, which must stabilise.
Now define $\nu\colon\bbN\to\bbN$ recursively by
  taking for $\nu(0)$ the smallest $i\in\bbN$ such that $x_j\le x_i$
  for infinitely many $j\in\bbN$. For $n> 0$ set \[\nu(n)=\min\{i >
  \nu(n-1)\mid x_j\le x_i\le x_{\nu(n-1)} \text{ for infinitely many
  }j\in\bbN\}.\qedhere\]
\end{proof}

The multiplication of paths in $Q$ yields the following notion of an
ideal. A non-empty set of paths $I$ in $Q$ is a \emph{left ideal} if
for any pair $\p,\psi$ of paths with $s(\psi)=t(\p)$ the condition
$\p\in I$ implies $\psi\p\in I$. A \emph{right ideal} of paths is
defined analogously.  Clearly, a subset $I\subseteq Q(x)$ is a poset
ideal if and only if it is a left ideal of paths.

\begin{defn}\label{de:noeth}
A quiver $Q$ is called \emph{left noetherian} if for each vertex $x$ the poset
$Q(x)$ is noetherian. This means the left ideals of paths starting at
any fixed vertex satisfy the ascending chain condition. The quiver is
\emph{right noetherian} if the right ideals of paths terminating at
any fixed vertex satisfy the ascending chain condition.
\end{defn}

This local condition can be checked by counting the number of
\emph{maximal paths} starting at a fixed vertex. By definition,
this is either a path $\a_n\ldots\a_2\a_1$ such that there is no arrow
starting at $t(\a_n)$, or it is a path $\ldots\a_2\a_1$ of
infinite length.

\begin{prop}\label{pr:quiver}
  Let $Q$ be a quiver. For a vertex $x$ the following conditions are equivalent.
\begin{enumerate}
\item The poset $Q(x)$ is noetherian.
\item For all  $\p\in Q(x)$ there are only finitely many
  arrows starting at $t(\p)$,  and for almost all $\p\in Q(x)$ there is at most one
 arrow starting at $t(\p)$.
\item There are only finitely many maximal paths
  starting at $x$.
\end{enumerate}
\end{prop}

\begin{proof}
  (1) $\Rightarrow$ (2): Let $\p\in Q(x)$.  The ideal of $Q(x)$ that is generated by all
  paths $\a \p$ with $\a$ an arrow starting at $t(\p)$ is finitely generated. It follows that there
  are only finitely many arrows starting at $t(\p)$. Suppose
  $\p_1,\p_2,\ldots$ is an infinite sequence of paths such that there
  is more than one arrow starting at each $t(\p_i)$. Because  $Q(x)$
  is noetherian, we may assume $\p_{i+1}\le \p_{i}$ for all $i$. For each
  $i$ we may choose $\a_i$ such that  $\p_{i+1}\not\le\a_i\p_i $. Then
the ideal of $Q(x)$ generated by  $\{ \a_i\p_i\mid i\in\bbN\}$ cannot
be generated by finitely many elements. This is a contradiction.

(2) $\Rightarrow$ (3): Let $n$ be the number of paths $\p\in Q(x)$ such
that there is more than one arrow starting at $t(\p)$, and let $d$ be
the maximal number of such arrows. Then there are at most $d^n$
maximal paths starting at $x$.

(3) $\Rightarrow$ (1):  Let $I\subseteq Q(x)$ be an
ideal. For each maximal path starting at $x$ and given by a (finite or infinite)
sequence of arrows
$\ldots\a_2\a_1$ we check if there exists $n\in\bbN$ such that $\a_n\ldots\a_2\a_1$
belongs to $I$. In that case we choose $n$ minimal, and it is easily
checked that $I$ is generated by all paths of this form. 
\end{proof}

Let us add another criterion for $Q(x)$ to be noetherian, following
\cite{HR1983, Ri1986}. We consider the \emph{path component} $Q_x$, so the
full subquiver of $Q$ with set of vertices $y\in Q_0$ that admit a
path $x\to y$. An \emph{oriented cycle} is a path $y\to y$ of length
at least one.

\begin{prop}\label{pr:quiver2}
  Let $Q$ be a quiver and $x$ a vertex. Then the poset $Q(x)$ is
  noetherian if and only if there are arrows $\a_1,\ldots,\a_n$ 
  such that there is a decomposition
  \begin{equation}\label{eq:decomp}
    Q_x\smallsetminus \{\a_1,\ldots,\a_n\}=Q^0\amalg Q^1\amalg
    \cdots\amalg Q^n
  \end{equation}
  into full and pairwise disjoint subquivers satisfying the following:
\begin{enumerate}
\item The quiver $Q^0$ has only finitely many vertices and arrows, and
  there is precisely one arrow starting at each vertex lying on an oriented cycle.
\item For $i\neq 0$ each quiver $Q^i$ is isomorphic
  to the infinite quiver \[\circ\lto\circ\lto\circ\lto\cdots.\]
\item Each arrow $\a_i$ 
  starts at a vertex of $Q^0$ not lying on an oriented cycle and terminates at the source of $Q^i$.
\end{enumerate}
\end{prop}

\begin{proof}
It is straightforward to check that the conditions on $Q_x$ are
equivalent to the condition on $Q(x)$ to have only finitely many
maximal paths starting at $x$. So the assertion follows from Proposition~\ref{pr:quiver}.
\end{proof}

An immediate consequence is the following simple criterion for quivers with finitely many vertices.

\begin{cor}\label{co:quiver-finite}
  Let $Q$ be quiver having only a finite number of vertices. Then $Q$ is
  left noetherian if and only if there are only finitely many arrows
  and there is precisely one arrow starting at each vertex lying on an
  oriented cycle.
\end{cor}
\begin{proof}
  We apply Proposition~\ref{pr:quiver2} and note that in the
  decomposition \eqref{eq:decomp} we have $Q_x=Q^0$, since for
  $i\neq 0$ the quiver $Q^i$ has infinitely many vertices.
\end{proof}

\section{Locally noetherian categories}

Let $\A$ be an abelian category. We assume that the subobjects of any
object form a complete lattice and that for every directed set of
subobjects $(X_i)_{i\in I}$ of an object $X$ and $Y\subseteq X$ one
has
\[\tag{AB5} \left(\sum_{i\in I}X_i\right)\cap Y=\sum_{i\in I}(X_i\cap
  Y).\] A collection of objects $(G_i)_{i\in I}$
\emph{generates} $\A$ if for each non-zero morphism $\p\colon X\to Y$
in $\A$ we have $\p\pi\neq 0$ for some morphism $\pi\colon G_i\to X$.
An object of an abelian category is called \emph{noetherian} if
the ascending chain condition holds for its subobjects. We write
$\Noeth\A$ for the full subcategory of noetherian objects in $\A$. 
An abelian category is \emph{cocomplete} if for any set of objects its coproduct exist.

\begin{defn}
  The abelian category $\A$ is \emph{noetherian} if every object is
  noetherian, and $\A$ is \emph{locally noetherian} if every object is
  the directed union of its noetherian subobjects.
\end{defn}

The following shows that our definition agrees with the usual one for
Grothendieck categories \cite[II.4]{Ga1962}.

\begin{lem}\label{le:locnoeth}
The category $\A$ is locally noetherian if and only if there is
 a collection of noetherian generators.
\end{lem}

\begin{proof}
Suppose that $\A$ is locally noetherian. Then we may take the noetherian
objects as the collection of generators, because for a morphism
$\p\colon X\to Y$ with $X=\sum_i X_i$ written as  the directed union of
its noetherian subobjects, we have
\[\Hom(X,Y)\cong\lim_i \Hom(X_i,Y).\]
If $\p\neq 0$, then one of the composites $X_i\hookrightarrow X\to Y$
needs to be non-zero.

Now suppose there is a collection of noetherian generators
$(G_i)_{i\in I}$. Let $X$ be an object and $X'\subseteq X$ the
directed union of its noetherian subobjects. It is clear that every composite
$G_i\to X\twoheadrightarrow X/X'$ is zero. Thus $X'=X$.
\end{proof}  

Let $Y\in\A$ be an object, written as the directed union $Y=\sum_j
Y_j$ of a set of subobjects. Then it is easily checked that for each
noetherian object $X\in\A$ the canonical map
\begin{equation}\label{eq:hom1}
  \colim_j\Hom(X,Y_j)\lto\Hom(X,\sum_j Y_j)
\end{equation}      
is a bijection. Clearly, the map is injective, and because $X$ is
noetherian any morphism $X\to Y$ factors through the inclusion
$Y_{j_0}\to Y$ for some index $j_0$. Here one uses the
condition (AB5). Thus the map is also surjective.

\begin{lem}
  Let $X=\sum_i X_i$ and $Y=\sum_j Y_j$ be objects in a locally
  noetherian category, written as directed union of their noetherian
  subobjects. Then there is a natural isomorphism of abelian groups
  \begin{equation}\label{eq:hom2}
    \lim_i\colim_j\Hom(X_i,Y_j)\longiso\Hom(X,Y).
\end{equation}    
\end{lem}
\begin{proof}
Combine \eqref{eq:hom1} with the fact that $\Hom(-,-)$ preserves
colimits in the first argument. 
\end{proof}

Now assume that $\Noeth\A$ is essentially small. We write
$\Lex((\Noeth\A)^\op,\Ab)$ for the category of additive contravariant functors
$\Noeth\A\to\Ab$ into the category of abelian groups that are left
exact and consider the \emph{restricted Yoneda functor}
\[\A\lto \bar\A:=\Lex((\Noeth\A)^\op,\Ab),\quad X\mapsto
  h_X:=\Hom(-,X)|_{\Noeth\A}.\]
Then the following is a slight generalisation of a theorem of Gabriel \cite{Ga1962}.

\begin{prop}\label{pr:locnoeth}
  Let $\A$ be a locally noetherian abelian category. Then the category $\Lex((\Noeth\A)^\op,\Ab)$ is a locally noetherian and
  cocomplete abelian category. The restricted Yoneda functor
  $\A\to\bar\A$ is fully faithful and exact. It restricts to an
  equivalence $\Noeth\A\iso\Noeth\bar\A$, and its essential image is
  closed under subobjects.
\end{prop}
\begin{proof}
  Let $X=\sum_i X_i$ and $Y=\sum_j Y_j$ be objects in $\A$, written as
  directed union of their noetherian subobjects.  Using
  \eqref{eq:hom1} we have
  \[h_X\cong\colim_i\Hom(-,X_i)\]
and compute
\begin{align*}
  \Hom(h_X,h_Y)&\cong\Hom(\colim_i\Hom(-,X_i),\colim_j\Hom(-,Y_j))\\
               &\cong\lim_i\Hom(\Hom(-,X_i),\colim_j\Hom(-,Y_j))\\
               &\cong  \lim_i\colim_j\Hom(X_i,Y_j)\\
               &\cong\Hom(X,Y).
\end{align*}
The second isomorphism uses that $\Hom(-,-)$ preserves colimits
in the first argument, the third follows from Yoneda's lemma, and the
last is \eqref{eq:hom2}. Thus $\A\to\bar\A$ is fully faithful.  For
the rest of the assertion, see \cite[II.4]{Ga1962}.
\end{proof}

\begin{exm}\label{ex:ring}
 Let $A$ be a ring. Then  the
regular module $A$ is a finitely generated generator for the category $\Mod A$ of left $A$-modules.
Thus $\Mod A$ is locally noetherian if and only if $A$ is left
noetherian \cite[V.4]{St1975}.
\end{exm}

\section{Representations of noetherian quivers}

Let $Q$ be a quiver and $\Q$ its path category. A \emph{representation} of
$Q$ in a category $\A$ is by definition a functor $\Q\to
\A$. Morphisms between representations are the natural
transformations, and we write $\Rep(Q,\A)$ for the category of representations.

Let $\A$ be a cocomplete abelian category. We assume that the
subobjects of any object form a complete lattice and that the
condition (AB5) holds. Given an object $M\in\A$, the functor
$\Hom(M,-)\colon\A\to\Set$ admits a left adjoint, which takes a set $X$
to the coproduct $M[X]$ of copies of $M$
indexed by the elements of $X$. Thus for each object $N\in\A$ there is
a natural isomorphism
\[\Hom(M[X],N)\cong\Hom(X,\Hom(M,N)).\]

Given a vertex  $x\in Q_0$, we write $\one_x$ for the subquiver of $Q$
consisting of the single vertex $x$ and no arrows.. Restriction along the inclusion $\one_x\hookrightarrow Q$
yields the evaluation functor
\[\Rep(Q,\A)\lto\Rep(\one_x,\A)=\A,\quad M\mapsto  M(x).\]
The left Kan extension provides a left adjoint, which is given by
\[\A\lto \Rep(Q,\A),\quad N\mapsto  N[Q(x,-)].\]
Thus we have for objects $N\in\A$ and $M\in \Rep(Q,\A)$ a natural
isomorphism
\begin{equation}\label{eq:adj}
\Hom(N[Q(x,-)],M)\cong\Hom(N,M(x)).
\end{equation}

\begin{lem}\label{le:gen}
Let $(M_i)_{i\in I}$ be a collection of generators of $\A$. Then the representations
$M_i[Q(x,-)]$ with $i\in I$ and $x\in Q_0$ generate $\Rep(Q,\A)$.
\end{lem}
\begin{proof}
This follows from the adjunction isomorphism \eqref{eq:adj}.
\end{proof}

An object $X$ is \emph{finitely generated} if for any directed set
of subobjects $(X_i)_{i\in I}$ with $X=\sum_{i\in I}X_i$ there is
$i_0\in I$ such that $X=X_{i_0}$.

\begin{lem}\label{le:fg}
If $M\in\A$ is finitely generated, then $M[Q(x,-)]$ is finitely
generated in $\Rep(Q,\A)$ for each $x\in Q_0$.
\end{lem}
\begin{proof}
Let  $M[Q(x,-)]=\sum_{i\in I}N_i$ be a directed union
of subobjects $N_i$. Then the adjunction isomorphism
\eqref{eq:adj} yields a morphism $M\to \sum_{i\in I}N_i(x)$, which
factors through $N_{i_0}(x)\to \sum_{i\in I}N_i(x)$ for some $i_0\in
I$ since $M$ is
finitely generated. Thus the identity  $M[Q(x,-)]\to\sum_{i\in I}N_i$
factors through $N_{i_0}\to\sum_{i\in I}N_i$, so  $M[Q(x,-)]=N_{i_0}$.
\end{proof}

Let $Q$ be a left noetherian quiver and fix $x\in Q_0$. We choose a
total order $\preceq$ on the countable set of arrows in $Q$ that lie on
a path in $Q(x)$. This order extends lexicographically to a total
order $\preceq$ on $Q(x)$ such that $\p <\psi$ implies
$\p\prec \psi$. Note that $(Q(x),\preceq)$ is noetherian, because any
ideal with respect to $\preceq$ is also one with respect to $\le$.

\begin{lem}\label{le:free}
  Let $Q$ be a left noetherian quiver and $x$ a vertex. If $M\in\A$ is a
  noetherian object, then the representation $M[Q(x,-)]$ is
  noetherian.
\end{lem}
\begin{proof}
Let $U\subseteq M[Q(x,-)]$ be a subrepresentation and $\p\in Q(x,y)$ a
path. We set \[Q(x,y)_\p:=\{\psi\in Q(x,y)\mid \p\preceq \psi\}\] and write
$U_\p$ for the image of the composite
\[U(y)\cap M[Q(x,y)_\p]\hookrightarrow M[Q(x,y)_\p]\twoheadrightarrow M,\] where
the second map denotes the projection corresponding to $\p$. For
$\psi\in Q(x,y)$, we have
\begin{equation}\label{eq:compos}
  \p \leq \psi\quad\implies\quad U_\psi\subseteq U_\p.
\end{equation}
More precisely, if there is a
path $\chi$ such that $\p=\chi\psi$, then composition with $\chi$
induces a map
\[Q(x,y)_\p\lto Q(x,y)_\psi\]
such that $\p$ is the only element that is mapped to $\psi$.
Thus $U(\chi)$ maps $U_\psi$ into $U_\p$.

Given subrepresentations $U\subseteq V\subseteq M[Q(x,-)]$, we have
$U_\p\subseteq V_\p$ for all $\p\in Q(x)$. If $U\neq V$, we claim
there exists $\p\in Q(x)$ such that $U_\p\neq V_\p$. To see this,
choose $y\in Q_0$ such that $U(y)\neq V(y)$.  Observe that the union
\begin{equation}\label{eq:dir1}
  Q(x,y)=\bigcup_{\p\in Q(x,y)}Q(x,y)_\p
\end{equation}
is directed, and therefore condition (AB5) implies
\[U(y)=\sum_{\p\in Q(x,y)}U(y)\cap M[Q(x,y)_\p].\] Thus there exists 
$\p\in Q(x,y)$ with
\[U(y)\cap M[Q(x,y)_\p]\neq V(y)\cap M[Q(x,y)_\p].\] Using that $Q(x)$
is noetherian we may choose $\p$ maximal with respect to
$\preceq$. Next observe that
\begin{equation}\label{eq:dir2}
  Q(x,y)_\p\smallsetminus\{\p\}=\bigcup_{\p\prec\psi}Q(x,y)_\psi
\end{equation}
is a directed union. This yields the following commutative diagram
with exact rows, using again the condition (AB5):
\[\begin{tikzcd}
    0\arrow{r}&\sum_{\p\prec \psi}U(y) \cap
    M[Q(x,y)_\psi]\arrow{r}\arrow[hookrightarrow]{d}&
    U(y)\cap M[Q(x,y)_\p]\arrow{r}\arrow[hookrightarrow]{d}&U_\p\arrow{r}\arrow[hookrightarrow]{d}&0 \\
    0\arrow{r}&\sum_{\p\prec\psi}M[Q(x,y)_\psi]\arrow{r}&M[Q(x,y)_\p]\arrow{r}&M\arrow{r}&0
  \end{tikzcd}\]
For $\p\prec\psi$ the maximality of $\p$ implies
\[U(y)\cap M[Q(x,y)_\psi]= V(y)\cap M[Q(x,y)_\psi]\] 
and therefore $U_\p\neq V_\p$.

Now suppose there is a proper ascending chain
\[U^1\subseteq U^2 \subseteq U^3\subseteq \cdots\] of
subrepresentations of $M[Q(x,-)]$. Pick for each $i\in\bbN$ a path
$\p_i$ such that $U^i_{\p_i}\subsetneq U^{i+1}_{\p_i}$. The set
$\{\p_i\mid i\in\bbN\}$ is infinite since $M$ is noetherian. So we may
assume $\p_i\neq \p_j $ for all $i\neq j$ in $\bbN$. Because  $Q(x)$ is
noetherian there exists a map $\nu\colon\bbN\to \bbN$ such that
$i<j$ implies $\nu(i)<\nu(j)$ and $\p_{\nu(j)}\leq \p_{\nu(i)}$; see
Lemma~\ref{le:poset}. Thus for all $i\in\bbN$ we have
\[U^{\nu(i)}_{\p_{\nu(i)}}\subsetneq U^{\nu(i)+1}_{\p_{\nu(i)}}
  \subseteq U^{\nu(i+1)}_{\p_{\nu(i)}}\subseteq
  U^{\nu(i+1)}_{\p_{\nu(i+1)}}\] where the last inclusion follows
from \eqref{eq:compos}.  This infinite chain is impossible since $M$
is noetherian, and therefore the representation $M[Q(x,-)]$ is
noetherian.
\end{proof}

Let $x\in Q_0$ and $A$ a ring. For any ideal $I\subseteq Q(x)$ and $y\in Q_0$ set
$I(x,y)=I\cap Q(x,y)$. Then $A[I(x,-)]$ yields a subrepresentation
of $A[Q(x,-)]$ in $\Rep(Q,\Mod A)$.

\begin{lem}\label{le:ring}
  Let $x\in Q_0$. The assignment $I\mapsto A[I(x,-)]$ induces an
  embedding of posets from the ideals of $Q(x)$ into the
  subrepresentations of $A[Q(x,-)]$.
\end{lem}
\begin{proof}
  This is clear since any inclusion of sets $X\subseteq Y$ induces an
  inclusion $A[X]\subseteq A[Y]$ of $A$-modules, with equality
  $A[X]=A[Y]$ only if $X=Y$.
\end{proof}

We are ready to prove our main result.

\begin{proof}[Proof of Theorem~\ref{th:repnoeth}]
  (1) $\Rightarrow$ (2): Let $M\in\Rep(Q,\A)$. We need to show that
  $M$ is the directed union of its noetherian subobjects. We may
  assume that $\A$ is essentially small, because it suffices to
  consider the full subcategory of $\A$ that is obtained by closing
  the set of objects $M(x)$ with $x\in Q_0$ under finite coproducts,
  subobjects, and quotients. Let $\bar\A$ denote the completion of
  $\A$, as in Proposition~\ref{pr:locnoeth}. The representations
  $N[Q(x,-)]$ with $N\in\Noeth\bar\A$ and $x\in Q_0$ are noetherian
  and generate $\Rep(Q,\bar\A)$ by Lemmas~\ref{le:gen} and
  \ref{le:free}. Thus $\Rep(Q,\bar\A)$ is locally noetherian; see
  Lemma~\ref{le:locnoeth}.  It follows that $M$ is the directed union
  of its noetherian subobjects, because we can view $\Rep(Q,\A)$ as a
  full subcategory which is closed under subobjects.
  
(2) $\Rightarrow$ (3): This is clear since  $\Mod A$ is locally noetherian for every 
left noetherian ring $A$; see Example~\ref{ex:ring}.

(3) $\Rightarrow$ (1): Fix $x\in Q_0$. The regular module $A$ is a finitely generated
object in $\Mod A$, and therefore the representation
$A[Q(x,-)]$ is finitely generated in $\Rep(Q,\Mod A)$ by
Lemma~\ref{le:fg}. It follows that $A[Q(x,-)]$ is noetherian
since $\Rep(Q,\Mod A)$ is locally noetherian, and we conclude from
Lemma~\ref{le:ring} that the poset $Q(x)$ is noetherian.
\end{proof}

\section{Modules over path algebras}

Let $Q$ be a quiver with a finite set of vertices and $A$ a ring.  The
\emph{path algebra} is by definition the
$A$-module \[A[Q]:=\coprod_{x,y\in Q_0}A[Q(x,y)]\] with multiplication
induced by the composition of paths in $Q$.  From the isomorphism
\eqref{eq:adj} it follows that $A[Q]$ identifies with the endomorphism
ring of the representation $\coprod_{x\in Q_0}A[Q(x,-)]$, and that the
assignment $M\mapsto \coprod_{x\in Q_0}M(x)$ induces an equivalence
\begin{equation}\label{eq:pathalg}
  \Rep(Q,\Mod A)\longiso\Mod A[Q].
\end{equation}

The following is well known, at least for path algebras over a field \cite[§1]{Cr1992}.

\begin{cor}\label{co:algebra}
  Let $Q$ be a quiver with a finite set of vertices and $A$ a left
  noetherian ring. Then the path algebra $A[Q]$ is left noetherian if
  and only if $Q$ has only finitely many arrows and there is
  precisely one arrow starting at each vertex lying on an oriented
  cycle.
\end{cor}
\begin{proof}
  The assertion follows from Theorem~\ref{th:repnoeth} and
  Corollary~\ref{co:quiver-finite}, because of the
  equivalence \eqref{eq:pathalg}.
\end{proof}

\section{Locally finite representations}

Let $\A$ be an abelian category. We assume that the
subobjects of any object form a complete lattice and that the
condition (AB5) holds. Recall that an object is of \emph{finite
  length} it it admits a finite composition series. An equivalent
condition is that both chain conditions hold for subobjects.

\begin{defn}
  The abelian category $\A$ is a \emph{length category} if every object is
  of finite length, and $\A$ is \emph{locally finite} if every object is
  the directed union of its finite length subobjects.
\end{defn}

\begin{defn}
  A quiver is called \emph{left finite} if for each vertex $x$ the set
  of paths starting at $x$ is finite. A quiver is \emph{right finite}
  if for each vertex $x$ the set of paths terminating at $x$ is
  finite.
\end{defn}

Using Proposition~\ref{pr:quiver} it is easily checked that $Q(x)$ is
left finite if and only if both chain conditions hold for ideals of
$Q(x)$. This observation suggests the following analogue of
Theorem~\ref{th:repnoeth}.

\begin{thm}\label{th:locfinite}
 For a quiver $Q$ and a left artinian ring $A$ the following conditions are equivalent.
\begin{enumerate}
\item The quiver $Q$ is left finite.
\item The abelian category $\Rep(Q,\A)$ is locally finite for every 
  locally finite abelian category $\A$.
\item The abelian category $\Rep(Q,\Mod A)$ is locally finite. 
\end{enumerate}
\end{thm}

\begin{proof}
  Recall that a ring is left artinian if and only if its category of
  left modules is locally finite \cite[VIII.1]{St1975}. On the other
  hand, the quiver $Q$ is left finite if and only if both chain conditions
  hold for the ideals of $Q(x)$ for all vertices $x$. Thus we can
  adapt the proof of Theorem~\ref{th:repnoeth}.
\end{proof}

\section{Gröbner enrichments}

Let $\C$ be a \emph{small category}. Thus the objects in $\C$ form a
set. For a pair of objects $x,y$ let
$\C(x,y)$ denote the set of morphisms $x\to y$ and let $\C(x)$ denote
the set of all morphisms with domain $x$. A partial order $\preceq$ on
$\C(x)$ is called a \emph{Gröbner enrichment} if the following holds
for each $y\in\C$:
\begin{enumerate}
  \item[(G1)] $\p\prec \psi$ implies $\omega\p\prec \omega\psi$ for all $\p,\psi\in 
    \C(x,y)$ and $\omega\in \C(y)$.
\item[(G2)] The poset $\C(x,y)$ is totally ordered and every
  non-empty subset has a maximal element.
\end{enumerate}

For any abelian category $\A$ let $\Rep(\C,\A)$ denote the category of
functors $\C\to\A$. The following result generalises Theorem~\ref{th:repnoeth}.

\begin{thm}\label{th:repnoeth-general}
  Let $\C$ be a small category and $A$ a left noetherian ring. Suppose
  there exists a Gröbner enrichment for each object in $\C$. Then the
  following conditions are equivalent.
\begin{enumerate}
\item The left ideals of morphisms starting at any fixed object satisfy the ascending chain condition.
\item The abelian category $\Rep(\C,\A)$ is locally noetherian for every 
  locally noetherian abelian category $\A$.
\item The abelian category $\Rep(\C,\Mod A)$ is locally noetherian. 
\end{enumerate}
\end{thm}
\begin{proof}
  The special case that $\C$ is the path category of a quiver is
  Theorem~\ref{th:repnoeth}. In fact, the Gröbner enrichment for a
  path category is specified right before Lemma~\ref{le:free}.  The
  proof of the general case is essentially the same.  The crucial
  ingredient is Lemma~\ref{le:free} where the Gröbner enrichment is
  used; its definition is extracted from the proof of that lemma. For
  instance, condition (G1) is needed for \eqref{eq:compos}, and the
  condition (G2) ensures that in \eqref{eq:dir1} and \eqref{eq:dir2}
  the union is directed. 
\end{proof}

\begin{rem}
  The above theorem is a variation of the main result of
  \cite{Ri1986}.  This was rediscovered by Sam and Snowden in
  \cite{SS2017} and they provide many interesting examples.  Their
  notion of a \emph{Gröbner category} and our notion of a Gröbner
  enrichment basically agree, except that they require each poset
  $\C(x,y)$ to be \emph{well-ordered}. Passing from $\preceq$ to its
  opposite shows that both concepts are equivalent.  The work in
  \cite{Ri1986,SS2017} restricts to representations in module
  categories; so these proofs are different from the one given here.
\end{rem}

\begin{rem}
  The above definition of a Gröbner enrichment for every object in
  $\C$ means that the category $\C$ is enriched in the category of
  posets satisfying condition (G2), with morphisms the strictly
  increasing maps.
\end{rem}

\begin{rem}
  We write $\p\le\psi$ for $\p,\psi\in \C(x)$ if $\p=\chi\psi$ for
  some morphism $\chi$. In many examples the partial order $\preceq$ 
  is a \emph{refinement} of $\le$, which means
\[\p\le\psi\quad\implies\quad\p\preceq\psi.\]
In that case the condition in (G2) that every non-empty subset has a
maximal element is automatic when (1) in
Theorem~\ref{th:repnoeth-general} holds. So it seems natural to impose
this condition in (G2).
\end{rem}

\begin{rem}
  The above condition (G2) can be relaxed. We need that for every
collection $\X$ of subsets of $\C(x,y)$ that is closed under taking subsets and
directed unions the following holds: if $\C(x,y)\not\in\X$, there exists
$\p\in\C(x,y)$ such that
\[\{\psi\in \C(x,y)\mid \p\prec \psi\}\in\X\quad \text{and}\quad\{\psi\in \C(x,y)\mid \p\preceq \psi\}\not\in\X.\] 
\end{rem}

\subsection*{Acknowledgements} 
This work was supported by the Deutsche
Forschungsgemeinschaft (SFB-TRR 358/1 2023 - 491392403).

\end{document}